\documentclass[11pt]{article}
\usepackage{amsmath,amsmath,amsthm,amssymb}
\usepackage{courier}
\usepackage{graphicx}
\usepackage{enumerate}
\usepackage{hyperref}

\newtheorem{theorem}{Theorem}

\newtheorem{corollary}{Corollary}
\newtheorem{proposition}{Proposition}
\newtheorem*{remark}{Remark}        
\numberwithin{equation}{section}

\newcommand{\ds}{\displaystyle}

\newcommand{\dv}{\text{div}}

\begin{document}

\title{Asymptotic and optimal Liouville properties for Wolff type integral systems}

\author{John Villavert\footnote{email: john.villavert@gmail.com, john.villavert@utrgv.edu} \\
[0.2cm] {\small School of Mathematical and Statistical Sciences}\\
{\small University of Texas Rio Grande Valley}\\
{\small Edinburg, Texas 78539 USA}
}
 
\date{}

\maketitle

\begin{abstract} 
This article examines the properties of positive solutions to fully nonlinear systems of integral equations involving Hardy and Wolff potentials. The first part of the paper establishes an optimal existence result and a Liouville type theorem for the integral systems. Then, the second part examines the decay rates of positive bound states at infinity. In particular, a complete characterization of the asymptotic properties of bounded and decaying solutions is given by showing that such solutions vanish at infinity with two principle rates: the slow decay rates and the fast decay rates. In fact, the two rates can be fully distinguished by an integrability criterion. As an application, the results are shown to carry over to certain systems of quasilinear equations. 
\end{abstract}

 \medskip


\section{Introduction}\label{Introduction}

This article studies the following class of fully nonlinear systems of integral equations with variable coefficients involving Hardy terms and Wolff potentials,
\begin{equation}\label{Wolff}
  \left\{\begin{array}{l}
    u(x) = c_{1}(x)W_{\beta,\gamma}(|y|^{\sigma_1}v^q)(x),\,~ x \in \mathbb{R}^n, \medskip \\
    v(x) = c_{2}(x)W_{\beta,\gamma}(|y|^{\sigma_2}u^p)(x),\,~ x \in \mathbb{R}^n.
  \end{array}
\right.
\end{equation}
The Wolff potential of a non-negative Borel measure $\mu$ is defined by
\begin{equation*}
W_{\beta,\gamma}(\mu) = \int_{0}^{\infty} \Big( \frac{\mu(B_{t}(x))}{t^{n-\beta\gamma}} \Big)^{\frac{1}{\gamma - 1}} \,\frac{dt}{t},
\end{equation*}
where $x \in \mathbb{R}^n$, $n \geq 3$, $\gamma > 1$, $\beta > 0$, $\beta\gamma < n$ and $B_{t}(x) \subset \mathbb{R}^n$ is the open ball of radius $t$ centered at $x$. Thus, if $d\mu = f \,dx$ where $f \in L^{1}_{loc}(\mathbb{R}^n)$ and $f \geq 0$, then 
\begin{equation*}
W_{\beta,\gamma}(f)(x) = \int_{0}^{\infty} \Big( \frac{\int_{B_{t}(x)} f(y) \,dy}{t^{n-\beta\gamma}} \Big)^{\frac{1}{\gamma - 1}} \,\frac{dt}{t}.
\end{equation*}

{\bf Convention:} Unless further specified, when considering system \eqref{Wolff} we always assume the following:
\begin{equation}\label{convention}
  \left\{\begin{array}{l}
 p,q > 0,\, \gamma \in (1,2], \, \sigma_i \text{ belongs to } (-\beta\gamma, \infty), \\
 \text{and the coefficients } c_{i}(x) \text{ are double bounded functions. }
  \end{array}
\right.
\end{equation}
Here we say a function $c(x)$ is a double bounded function if there exists a positive constant $C$ such that $1/C \leq c(x) \leq C$ for $x \in \mathbb{R}^n$. We say $(u,v)$ is a {\bf solution} of system \eqref{Wolff} if $u,v \in L^{1}_{loc}(\mathbb{R}^n)$ are non-negative and satisfy the integral equations for a.e. $x \in \mathbb{R}^n$. In addition, for a given positive solution, we say it is a {\bf decaying solution} if there exist positive rates $\theta_1$ and $\theta_2$ such that $u(x) \simeq |x|^{-\theta_1}$ and $v(x) \simeq |x|^{-\theta_2}$. Here, the notation $f(x) \simeq |x|^{-\theta}$ means there exists a positive constant $c$ such that $1/c \leq |x|^{\theta}f(x) \leq c$ as $|x| \rightarrow \infty$.

The goals of this paper are to establish some new results concerning the optimal existence and non-existence of positive solutions and to continue the study from \cite{Villavert:14e} on the decay properties of positive solutions for system \eqref{Wolff}. Our study on the integral equations involving the Wolff potential has roots in the qualitative analysis of elliptic partial differential equations arising from nonlinear analysis, calculus of variations, conformal geometry and mathematical physics. The prototypical example is the semilinear equation with weight,
\begin{equation}\label{weighted PDE}
 -\Delta u = |x|^{\sigma}u^p,\, u > 0 \, \text{ in } \mathbb{R}^n,
 \end{equation}
where $p > 1$ and $\sigma \in \mathbb{R}$, and we illustrate soon below how this equation is equivalent to a very simple case of system \eqref{Wolff}. Indeed, equation \eqref{weighted PDE} arises as an important stationary model for stellar cluster formation in astrophysics \cite{Henon73}. Liouville type theorems for \eqref{weighted PDE}, when combined with blow-up and rescaling arguments, lead to a priori estimates to a family of elliptic boundary value problems \cite{GS81a}, and the classification of its solutions when $\sigma = 0$ and $p = (n+2)/(n-2)$ also plays an important role in the Yamabe and prescribing scalar curvature problems. It turns out that the decay properties of solutions to equation \eqref{weighted PDE} are closely related to these other properties, and we have a fairly good picture of the asymptotic behavior of the bound states. We outline such results here for completeness sake and because the main theorems of this paper can be viewed as generalized versions for the Wolff type integral systems. It is known that equation \eqref{weighted PDE} has no solution if $\sigma \leq -2$ or when $\sigma > -2$ and $1 < p \leq \frac{n+\sigma}{n-2}$ (see \cite{GS81,MitidieriPohozaev01}). Thus, if a solution exists, then $p > \frac{n+\sigma}{n-2}$ and $\sigma > -2$ necessarily hold. This implies that $\frac{2+\sigma}{p-1} < n-2$ and this motivates the following terminology for the fast and slow decaying solutions. It is known that bound state solutions which vanish at infinity must do so with two principle rates of decay: the fast rate $u(x) \simeq |x|^{-(n-2)}$ or the slow rate $u(x) \simeq |x|^{-\frac{2+\sigma}{p-1}}$  \cite{YLi92} (see also \cite{CGS89,DdPMW08,Lei2013} and the references therein).

Let us make the connection between elliptic partial differential equations and the Wolff type integral equations more apparent. To do so, we first note that the Wolff potential $W_{\beta,\gamma}(\cdot)$ reduces to the well-known Riesz potential $I_{\alpha}(\cdot)$ multiplied by a positive constant when $\beta = 2$ and $\gamma = \alpha/2$. Namely, 
\begin{equation*}
W_{\frac{\alpha}{2},2}(f)(x) = \frac{1}{(n-\alpha)}\int_{\mathbb{R}^n} \frac{f(y)}{|x-y|^{n-\alpha}} \,dy =: \frac{1}{(n-\alpha)} I_{\alpha}(f)(x).
\end{equation*}
System \eqref{Wolff} under constant coefficients therefore includes a weighted variant of the Hardy-Littlewood-Sobolev (HLS) type integral system
\begin{equation}\label{wHLS}
  \left\{\begin{array}{l}
    u(x) = \ds\int_{\mathbb{R}^n} \frac{|y|^{\sigma_1}v^{q}(y)}{|x-y|^{n-\alpha}} \,dy, \medskip \\
    v(x) = \ds\int_{\mathbb{R}^n} \frac{|y|^{\sigma_2}u^{p}(y)}{|x-y|^{n-\alpha}} \,dy.
  \end{array}
\right.
\end{equation}
In the unweighted case, i.e., $\sigma_1,\sigma_2 = 0$, these are the Euler-Lagrange equations for a functional associated with the best constant in the HLS inequality \cite{Lieb83}. If $p = q$, $\sigma_1 = \sigma_2$ and $u \equiv v$, the integral system reduces to an integral equation, which is also associated with the Euler-Lagrange equation for the sharp Hardy-Sobolev inequality \cite{CKN84,LuZhu11,JYang15}. The HLS type integral systems are naturally associated with systems of differential equations. For example, when $\alpha = 2k$ is an even integer, $\sigma_1,\sigma_2 \leq 0$ and $p,q > 1$, system \eqref{wHLS} is equivalent to the poly-harmonic system of the H\'{e}non-Lane-Emden type \cite{CL13,Villavert:14c}:
\begin{equation}\label{wPDE}
  \left\{\begin{array}{cl}
    (-\Delta)^{k}u = |x|^{\sigma_1}v^{q}, \\
    (-\Delta)^{k}v = |x|^{\sigma_2}u^{p},
  \end{array}
\right.
\end{equation}
which reduces to equation \eqref{weighted PDE} if $k=1$, $\sigma = \sigma_1 = \sigma_2$, $p = q$ and $u \equiv v$. 

\begin{remark}
If $k = 1$, system \eqref{wPDE} is often called the H\'{e}non system when $\sigma_1,\sigma_2 > 0$, the Lane-Emden system when $\sigma_1,\sigma_2 = 0$, or the Hardy system when $\sigma_1,\sigma_2 < 0$, but we will just refer to it as the H\'{e}non-Lane-Emden system in any case.
\end{remark}
Similarly, if $\beta = 1$, system \eqref{Wolff} is closely related to the system of $\gamma$-Laplace equations of the H\'{e}non-Lane-Emden type
\begin{equation*}
  \left\{\begin{array}{cl}
    -\dv\,(|\nabla u|^{\gamma-2}\nabla u) = c_{1}(x)|x|^{\sigma_1}v^{q}, \medskip \\
    -\dv\,(|\nabla v|^{\gamma-2}\nabla v) = c_{2}(x)|x|^{\sigma_2}u^{p},
  \end{array}
\right.
\end{equation*}
and we shall describe their close relationship in more detail shortly below. Other relevant examples include more general quasilinear systems, including those involving $k$-Hessian operators (see \cite{Villavert:14e} and the references therein).

Just as we have for the prototypical elliptic equation, we show that bounded and decaying positive solutions of \eqref{Wolff} exhibit only two rates of decay: the fast decay rates and the slow decay rates. Here we say a decaying solution $(u,v)$ of system \eqref{Wolff} {\bf decays with the slow rates} as $|x| \rightarrow \infty$ if 
$$ u(x) \simeq |x|^{-q_0} \,\text{ and }\, v(x) \simeq |x|^{-p_0},$$
where
$$\textstyle q_0 = \frac{\beta\gamma(\gamma - 1 + q) + (\gamma - 1)\sigma_1 + \sigma_2 q}{pq - (\gamma-1)^2} \,\text{ and }\, p_0 = \frac{\beta\gamma(\gamma - 1 + p) + (\gamma - 1)\sigma_2 + \sigma_1 p}{pq - (\gamma-1)^2}.$$
On the other hand, we previously established the following equivalent characterization of the fast decaying solutions and whose definition is contained in the last statement of the theorem.
\begin{theorem}[Theorem 1 in \cite{Villavert:14e}]\label{fast theorem}
Let $\gamma \in (1,2]$, $q\geq p > 1$ and $\sigma_1 \leq \sigma_2 \leq 0$ and let $(u,v)$ be a positive solution of the integral system \eqref{Wolff} where 
\begin{equation}\label{non-subcritical}
q_0 + p_0 \leq \frac{n-\beta\gamma}{\gamma-1}.
\end{equation}
Then the following statements are equivalent.
\begin{enumerate}[(a)]
\item $(u,v) \in L^{r_0}(\mathbb{R}^n) \times L^{s_0}(\mathbb{R}^n)$, where
$$ r_0 = \frac{n}{q_0} \,\text{ and }\, s_0 = \frac{n}{p_0}.$$

\item $(u,v) \in L^{r}(\mathbb{R}^n) \times L^{s}(\mathbb{R}^n)$, where 
\begin{equation}\label{optimal integrability}
r > \frac{n(\gamma - 1)}{n-\beta\gamma} \,\text{ and }\, s > \max\Big\lbrace \frac{n(\gamma - 1)}{n-\beta\gamma},\, \frac{n(\gamma-1)}{p(\frac{n-\beta\gamma}{\gamma - 1}) - (\beta\gamma +\sigma_2)} \Big\rbrace.
\end{equation}
\item $(u,v)$ is bounded and decaying, and it {\bf decays with the fast rates} as $|x|\rightarrow \infty$, i.e.,
$$u(x) \simeq |x|^{-\frac{n-\beta\gamma}{\gamma-1}}$$
and
\begin{equation*}
v(x) \simeq \left\{\begin{array}{ll}
|x|^{-\frac{n-\beta\gamma}{\gamma-1}}, 							                     & \text{ if }\, p(\frac{n-\beta\gamma}{\gamma-1}) - \sigma_2 > n; \medskip \\
|x|^{-\frac{n-\beta\gamma}{\gamma-1}}(\ln |x|)^{\frac{1}{\gamma-1}},	                 & \text{ if }\, p(\frac{n-\beta\gamma}{\gamma-1}) - \sigma_2 = n; \\
|x|^{- \frac{p(\frac{n-\beta\gamma}{\gamma-1}) - (\beta\gamma + \sigma_2)}{\gamma-1}}, & \text{ if }\, p(\frac{n-\beta\gamma}{\gamma-1}) - \sigma_2 < n.
\end{array}
\right.
\end{equation*}
\end{enumerate}
\end{theorem}
In view of this, we say a solution $(u,v)$ of system \eqref{Wolff} is an {\bf integrable solution} if $(u,v) \in  L^{r_0}(\mathbb{R}^n) \times L^{s_0}(\mathbb{R}^n)$ and say it is an {\bf optimal integrable solution} if $(u,v) \in L^{r}(\mathbb{R}^n) \times L^{s}(\mathbb{R}^n)$ for all $(r,s)$ satisfying \eqref{optimal integrability}.

There are some important observations that should be made here. The previous theorem asserts that solutions, under a fairly mild integrability assumption, are indeed bounded and fast decaying. The assumptions that $q \geq p$ and $\sigma_1 \leq \sigma_2$ are due to the in-homogeneity of the system (an issue which does not occur in the scalar case), but they are not so crucial. More precisely, the theorem still holds if $p \geq q$ and $\sigma_2 \leq \sigma_1$ provided that these parameters along with $u$ and $v$ are interchanged in the statement of the theorem. Condition \eqref{non-subcritical}, which is equivalent to perhaps the more familiar condition
\begin{equation}\label{non-subcritical 1}
\frac{n+\sigma_1}{q + \gamma - 1} + \frac{n+\sigma_2}{p + \gamma - 1} \leq \frac{n-\beta\gamma}{\gamma - 1}
\end{equation}
when $pq > (\gamma - 1)^2$, is certainly stronger than the condition $\max\lbrace q_0, \, p_0 \rbrace < (n-\beta\gamma)/(\gamma - 1)$; however, making this stronger assumption is not without proper motivation. For instance, in the special case of system \eqref{wHLS} with $\alpha \in (1,n)$, it turns out that the system admits neither a bounded and decaying positive classical solution nor a positive integrable solution in the subcritical case \cite{Villavert:14d},
$$\frac{n+\sigma_1}{q + 1} + \frac{n+\sigma_2}{p+1} > n-\alpha.$$ 
In fact, when $\sigma_1,\sigma_2 = 0$, system \eqref{wHLS} admits a positive integrable solution if and only if the critical case holds \cite{LL13,Lieb83}, i.e., 
 $$\frac{1}{q + 1} + \frac{1}{p+1} = \frac{n-\alpha}{n}.$$ 
 Of course, we conjecture that the same holds true for the more general Wolff type integral systems, but a proof of this escapes us at this time. However, we do have a closely related result below (see Theorem \ref{Liouville}). Even for system \eqref{wHLS} such a result on the non-existence of positive classical solutions is quite non-trivial and it is often called the (generalized) HLS conjecture \cite{Caristi2008} (in the case of system \eqref{wPDE} with $k=1$, it is more commonly called the H\'{e}non-Lane-Emden conjecture). It is crucial to note that no boundedness or growth assumptions are being imposed here and the known partial results are limited to dimension $n\leq 4$ or under the aforementioned assumptions \cite{Mitidieri96,Phan12,PQS07,SZ96,Souplet09}. 
 
In addition, the intervals of integrability in \eqref{optimal integrability} are optimal. Namely, if a solution of \eqref{Wolff} belongs to $L^{r}(\mathbb{R}^n) \times L^{s}(\mathbb{R}^n)$, then necessarily (see \cite{Villavert:14e})
\begin{align*} 
r > {} & \max\Big\lbrace \frac{n(\gamma - 1)}{n-\beta\gamma},\, \frac{n(\gamma-1)}{q(\frac{n-\beta\gamma}{\gamma - 1}) - (\beta\gamma +\sigma_1)} \Big\rbrace \,\text{ and }\, \\
s > {} & \max\Big\lbrace \frac{n(\gamma - 1)}{n-\beta\gamma},\, \frac{n(\gamma-1)}{p(\frac{n-\beta\gamma}{\gamma - 1}) - (\beta\gamma +\sigma_2)} \Big\rbrace.
\end{align*}
Therefore, $q \geq p$, $\sigma_1 \leq \sigma_2$ and \eqref{non-subcritical} imply that
$$r > \max\Big\lbrace \frac{n(\gamma - 1)}{n-\beta\gamma},\, \frac{n(\gamma-1)}{q(\frac{n-\beta\gamma}{\gamma - 1}) - (\beta\gamma +\sigma_1)} \Big\rbrace = \frac{n(\gamma - 1)}{n-\beta\gamma}.$$
We also mention some previous articles concerning the analysis of the HLS and Wolff type integral systems, especially since they have ultimately inspired the work in this paper. The study on the existence, non-existence, classification and decay properties of HLS type and related systems can be found in \cite{Caristi2008,ChengLi15,CLO05,CLO06,DAmbrosioMitidieri14,DouZhou15,LL13a,LLM12,Villavert:14b}. Similar studies on the Wolff type integral systems can be found in \cite{ChenLu2014,CL11,Lei2011,LeiLi2012,PhucVerbitsky2008,SunLei2012}.

We are now ready to state the main results of this paper.

\subsection{Optimal Liouville theorem for positive solutions}

Our first main result for system \eqref{Wolff} is the following non-existence result.
\begin{theorem}\label{Liouville}
Let $\beta > 0$ and $\gamma > 1$ with $\beta\gamma < n$. Then the integral system \eqref{Wolff} has no positive solution for any double bounded coefficients $c_{1}(x)$ and $c_{2}(x)$ if either $$pq \leq (\gamma -1)^2 \,\text{ or }\, pq > (\gamma - 1)^2 \,\text{ and }\, \max \Big\lbrace q_0, p_0 \Big\rbrace > \frac{n-\beta\gamma}{\gamma-1}.$$
If, in particular, $\gamma \in (1,2]$, then the same conclusion holds in the endpoint case:
$$pq > (\gamma - 1)^2 \,\text{ and }\, \max \Big\lbrace q_0, p_0 \Big\rbrace = \frac{n-\beta\gamma}{\gamma-1}.$$
\end{theorem}

\begin{remark}
As a consequence of Theorem \ref{Liouville}, when given some positive solution $(u,v)$ to either system \eqref{Wolff} or the quasilinear systems considered below, we shall therefore assume that 
\begin{equation}\label{necessary existence}
pq > (\gamma-1)^2 \,\text{ and }\, \max \Big\lbrace q_0, p_0 \Big\rbrace < \frac{n-\beta\gamma}{\gamma-1}.
\end{equation}
The next theorem indicates that, under the assumptions in \eqref{convention}, Theorem \ref{Liouville} is indeed sharp.
\end{remark}

\begin{theorem}\label{theorem2}
Suppose that $p,q,\sigma_{1},\sigma_{2},\beta$ and $\gamma$ satisfy condition \eqref{necessary existence}. Then there exist double bounded coefficients $c_{1}(x)$ and $c_{2}(x)$ such that the integral system \eqref{Wolff} admits a positive solution.
\end{theorem}

\subsection{Decay rates of positive solutions}

The following results concern the decay properties of bounded and decaying solutions starting with a result on the slow decaying solutions.

\begin{theorem}\label{slow theorem}
Let $\beta > 0$ and $\gamma > 1$ with $\beta\gamma < n$ and suppose $(u,v)$ is a bounded positive solution of \eqref{Wolff}. Then the following statements hold.
\begin{enumerate}[(i)]
\item If $\theta_{1} < q_0$ and $\theta_{2} < p_0$, then there does not exist any positive constant $c$ for which either
$$ u(x) \geq c(1 + |x|)^{-\theta_1} \,\text{ or }\, v(x) \geq c(1+|x|)^{-\theta_2} \text{ for a.e. } x \in \mathbb{R}^n.$$

\item If $\theta_3 > q_0$, $\theta_4 > p_0$ and $(u,v)$ is not integrable, i.e., either $u \not\in L^{r_0}(\mathbb{R}^n)$ or $v \not\in L^{s_0}(\mathbb{R}^n)$, then there does not exist any positive constant $C$ for which either
$$ u(x) \leq C(1 + |x|)^{-\theta_3} \,\text{ or }\, v(x) \leq C(1+|x|)^{-\theta_4} \text{ for a.e. } x \in \mathbb{R}^n.$$
\item If $(u,v)$ is a decaying solution but is not integrable, then it necessarily decays with the slow rates as $|x| \rightarrow \infty$, i.e.,
$$ u(x) \simeq |x|^{-q_0} \,\text{ and }\, v(x) \simeq |x|^{-p_0}.$$
\end{enumerate}
\end{theorem}
Parts (i) and (ii) of the theorem, in a sense, imply that bounded positive solutions that are not integrable have almost the slow rates. Then part (iii) demonstrates that if it is also decaying, then it actually has the slow rates. So this result naturally complements Theorem \ref{fast theorem}, however, Theorem \ref{slow theorem}  is far more general than Theorem \ref{fast theorem} in terms of the assumptions placed on system \eqref{Wolff}, namely, on the parameters $\gamma,p,q,\sigma_1$ and $\sigma_2$. This leads us to ask if we may establish another version of Theorem \ref{fast theorem} that relaxes the assumptions on the parameters. In view of this, we do have such a result, however, we must restrict our attention to the optimal integrable solutions.
\begin{theorem}\label{improved fast theorem}
Let $\beta > 0$ and $\gamma > 1$ with $\beta\gamma < n$, $q\geq p$, $\sigma_1 \leq \sigma_2$ and suppose $(u,v)$ is a bounded and decaying positive solution of system \eqref{Wolff} satisfying \eqref{non-subcritical}.
The following statements are equivalent.
\begin{enumerate}[(i)]
\item $(u,v)$ is an optimal integrable solution. 

\item $(u,v)$ decays with the fast rates as $|x|\rightarrow \infty$.
\end{enumerate}
\end{theorem}

\begin{remark}
Under the assumptions of Theorem \ref{fast theorem}, the optimal integrable solutions are equivalent to the integrable solutions. The proof of this relies on key $L^p$ comparison estimates between the Riesz and Wolff potentials, the weighted HLS inequality, and a delicate bootstrap or lifting technique. However, we do not know if this equivalence remains true under the more general conditions of Theorem \ref{improved fast theorem}, even under the additional assumption that the solution is bounded and decaying.
\end{remark}

\subsection{Quasilinear systems}
In establishing the asymptotic and Liouville type results for our family of integral systems, we also obtain analogous results for the general quasilinear system of the form
\begin{equation}\label{quasilinear system}
  \left\{\begin{array}{cl}
    -\dv\,\mathcal{A}(x,\nabla u) = c_{1}(x)|x|^{\sigma_1}v^{q}, \medskip \\
    -\dv\,\mathcal{A}(x,\nabla v) = c_{2}(x)|x|^{\sigma_2}u^{p},
  \end{array}
\right.
\end{equation}
where the map $\mathcal{A}:\mathbb{R}^n \times \mathbb{R}^n \mapsto \mathbb{R}^n$ satisfies the following properties. The mapping $x \mapsto \mathcal{A}(x,\xi)$ is measurable for all $\xi \in \mathbb{R}^n;$ the mapping $\xi \mapsto \mathcal{A}(x,\xi)$ is continuous for a.e. $x \in \mathbb{R}^n$; for some positive constants $k_{1}\leq k_{2}$ there hold for all $\xi\in\mathbb{R}^n$ and a.e. $x \in \mathbb{R}^n$,
\begin{enumerate}[(a)]
\item $\mathcal{A}(x,\xi)\cdot\xi \geq k_{1}|\xi|^{\gamma}$,
\item $|\mathcal{A}(x,\xi)|\leq k_{2}|\xi|^{\gamma-1}$,
\item $(\mathcal{A}(x,\xi) - \mathcal{A}(x,\xi'))\cdot(\xi - \xi') > 0$ whenever $\xi \neq \xi'$,
\item $\mathcal{A}(x,\lambda\xi) = \lambda|\lambda|^{\gamma-2}\mathcal{A}(x,\xi)$ for all $\lambda \neq 0$.
\end{enumerate}
We say $(u,v)$ is a (weak) solution of \eqref{quasilinear system} if $u$ and $v$ belong to $W^{1,\gamma}_{loc}(\mathbb{R}^n)\cap C(\mathbb{R}^n)$ and satisfy the equations in the distribution sense. In the simple case where $\mathcal{A}(x,\xi) \doteq |\xi|^{\gamma-2}\xi$, $\dv\,\mathcal{A}(x,\nabla u)$ is just the classical $\gamma$-Laplace operator.

To illustrate the relationship between quasilinear operators and Wolff potentials, we recall a consequence of the global pointwise estimates of \cite{KM94}[Corollary 4.13]. Namely, if $(u,v)$ is a positive solution of \eqref{quasilinear system} satisfying
$$ \inf_{\mathbb{R}^n} u = \inf_{\mathbb{R}^n} v = 0,$$
then there exist positive constants $C_1$ and $C_2$, depending only on $n$ and $\gamma$ and the structural constants $k_1$ and $k_2$, such that 
\begin{equation}\label{KM}
  \left\{\begin{array}{cl}
C_{1}W_{1,\gamma}(c_{1}(y)|y|^{\sigma_1} v^q)(x) \leq u(x) \leq C_{2}W_{1,\gamma}(c_{1}(y)|y|^{\sigma_1} v^q)(x), \\
C_{1}W_{1,\gamma}(c_{2}(y)|y|^{\sigma_2} u^p)(x) \leq v(x) \leq C_{2}W_{1,\gamma}(c_{2}(y)|y|^{\sigma_2} u^p)(x).
  \end{array}
\right.
\end{equation}
In view of this and by applying our results on the Wolff type integral systems, we shall establish the following.
\begin{corollary}\label{quasilinear Liouville}
Let $\beta = 1$ and $\gamma \in (1,2]$. For any pair of double bounded coefficients $c_{1}(x)$ and $c_{2}(x)$, system \eqref{quasilinear system} has no positive solution $(u,v)$ satisfying $$\inf_{\mathbb{R}^n} u = \inf_{\mathbb{R}^n} v = 0,$$
whenever $pq \in (0,(\gamma-1)^2]$ or if $pq > (\gamma-1)^2$ and
$$\textstyle \max\{ q_0,\, p_0 \} = \max\Big\{\frac{\gamma(\gamma-1 + q) + (\gamma - 1)\sigma_1 + \sigma_2 q}{pq - (\gamma-1)^2}, \frac{\gamma(\gamma -1 + p) + (\gamma - 1)\sigma_2 + \sigma_1 p}{pq - (\gamma-1)^2} \Big\} \geq \frac{n-\gamma}{\gamma-1}.$$
\end{corollary}

%

We have the following decay properties of solutions for the quasilinear systems. We begin with a corollary of Theorem \ref{fast theorem} for completeness sake.

\begin{corollary}[Corollary 2 in \cite{Villavert:14e}]\label{fast quasilinear}
Let $\beta = 1$, $\gamma \in (1,2]$, $q\geq p > 1$, $\sigma_1 \leq \sigma_2 \leq 0$ and let $(u,v)$ be a positive solution of system \eqref{quasilinear system} with $q_0 + p_0 \leq \frac{n-\gamma}{\gamma -1}$. Then $(u,v)$ is an integrable solution if and only if $(u,v)$ is bounded, decaying and decays with the fast rates as $|x|\rightarrow \infty$, i.e.,
$$u(x) \simeq |x|^{-\frac{n-\gamma}{\gamma-1}}$$
and
\begin{equation*}
v(x) \simeq \left\{\begin{array}{ll}
|x|^{-\frac{n-\gamma}{\gamma-1}}, 							                     & \text{ if }\, p(\frac{n-\gamma}{\gamma-1}) - \sigma_2 > n; \medskip \\
|x|^{-\frac{n-\gamma}{\gamma-1}}(\ln |x|)^{\frac{1}{\gamma-1}},	                 & \text{ if }\, p(\frac{n-\gamma}{\gamma-1}) - \sigma_2 = n; \\
|x|^{- \frac{p(\frac{n-\gamma}{\gamma-1}) - (\gamma + \sigma_2)}{\gamma-1}}, & \text{ if }\, p(\frac{n-\gamma}{\gamma-1}) - \sigma_2 < n.
\end{array}
\right.
\end{equation*}
 
\end{corollary}

\begin{corollary}\label{slow quasilinear}
Let $\beta = 1$, $\gamma \in (1,n)$ and let $(u,v)$ be a bounded and decaying positive solution of \eqref{quasilinear system}. If $(u,v)$ is not integrable, then it necessarily decays with the slow rates as $|x|\rightarrow \infty$, i.e.,
$$ u(x) \simeq |x|^{-q_0} \,\text{ and }\, v(x) \simeq |x|^{-p_0}.$$
\end{corollary}

\begin{corollary}\label{fast quasilinear improved}
Let $\beta = 1$, $\gamma \in (1,n)$, $q \geq p > 1$, $\sigma_1 \leq \sigma_2 $ and let $(u,v)$ be a bounded and decaying positive solution of system \eqref{quasilinear system} satisfying $q_{0} + p_{0} \leq \frac{n-\gamma}{\gamma - 1}$. Then $(u,v)$ is an optimal integrable solution if and only if $(u,v)$ decays with the fast rates as $|x|\rightarrow \infty$.
\end{corollary}

The remaining parts of this paper are organized as follows. In \S \ref{section2}, the proof of Theorem \ref{Liouville} is provided followed by the proof of Theorem \ref{theorem2}. Then, \S \ref{section3} and \S \ref{section4}, respectively, contains the proof of Theorem \ref{slow theorem} and Theorem \ref{improved fast theorem}. Lastly, \S \ref{section5} contain the proofs for the corresponding results on the quasilinear systems.  

\begin{remark}[Notation]
Throughout this paper, we adopt the standard convention that $c$, $C$, $C_{1}, C_{2}\ldots,$ are positive universal constants in the inequalities that may change from line to line (and sometimes within the same line itself).
\end{remark}

\section{Existence and Liouville property of solutions}\label{section2}

\subsection{Non-existence of positive solutions} We now prove our Liouville type theorem.
\begin{proof}[Proof of Theorem \ref{Liouville}]

We proceed by contradiction. That is, assume there is a positive solution $(u,v)$. Let $|x| > R$ for some suitable $R>0$ and note that       
$$0 < C_1 \leq \int_{B_{R}(0)} |y|^{\sigma_1}v^{q}(y) \,dy \leq C_2 < \infty.$$ 
Then, from the first integral equation there holds
\begin{equation*}
u(x) \geq C\int_{|x|+R}^{\infty} \Big( \frac{\int_{B_{R}(0)} |y|^{\sigma_1}v^{q}(y) \,dy}{t^{n-\beta\gamma}} \Big)^{\frac{1}{\gamma-1}} \,\frac{dt}{t} \geq C\int_{|x|+R}^{\infty} t^{-\frac{n-\beta\gamma}{\gamma-1}}\,\frac{dt}{t} \geq \frac{C}{|x|^{a_0}} ,
\end{equation*}
where $a_0 = (n-\beta\gamma)/(\gamma-1)$. Inserting this into the second integral equation yields for $|x|>R$, 
\begin{align*}
v(x) \geq {} & C\int_{2|x|}^{\infty}\Big( \int_{B_{t-|x|}(0)\backslash B_{\frac{t-|x|}{2}}(0)} \frac{dy}{|y|^{pa_0 - \sigma_2}}t^{-(n-\beta\gamma)} \Big)^{\frac{1}{\gamma - 1}}\,\frac{dt}{t} \\
\geq {} & C\int_{2|x|}^{\infty} t^{-\frac{p a_0 -\sigma_2 -\beta\gamma}{\gamma - 1}}\,\frac{dt}{t}.
\end{align*}
Now, if $pa_0 - \sigma_2 - \beta\gamma \leq 0$, then the previous estimate implies $v(x) = \infty$ and we arrive at the desired contradiction. Otherwise, if $pa_0 - \sigma_2 - \beta\gamma > 0$, we deduce instead 
\begin{equation*}
v(x) \geq C|x|^{- b_0}, ~\text{ for } |x| > R,
\end{equation*}
where $b_0 = (pa_0 - \sigma_2 - \beta\gamma)/(\gamma-1)$. Likewise, using the previous estimate, if $qb_0 - \sigma_1 - \beta\gamma \leq 0$ then $u(x) = \infty$; otherwise, if $qb_0 - \sigma_1 - \beta\gamma > 0$, inserting the last estimate into the first integral equation yields
\begin{equation*}
u(x) \geq C|x|^{- a_1}, ~\text{ for } |x| > R,
\end{equation*}
where $a_1 = (qb_0 - \sigma_1 - \beta\gamma)/(\gamma-1)$. Assuming that we can continue this procedure, we arrive at 
$$ u(x) \geq C|x|^{-a_k} \,\text{ and }\, v(x) \geq C|x|^{-b_k}, \,\text{ for }\, |x| > R,$$
where 
\begin{equation*}
b_k = \frac{pa_k - \sigma_2 - \beta\gamma}{\gamma-1} \,\text{ and }\, a_k = \frac{qb_{k-1} - \sigma_1 - \beta\gamma}{\gamma-1} \,\text{ for }\, k =1,2,3,\ldots.
\end{equation*}
Actually, using the definitions of $a_k$ and $b_k$, we calculate that
\begin{align}\label{sequence}
a_j = {} & \frac{q b_{j-1} -  \sigma_1 - \beta\gamma}{\gamma-1} = \frac{pq}{(\gamma-1)^2} a_{j-1} - \frac{\beta\gamma(\gamma-1+q) + (\gamma -1)\sigma_1 + q\sigma_2}{(\gamma-1)^2} \notag \\
= {} & \Big(\frac{pq}{(\gamma-1)^2}\Big)^2 a_{j-2} - \frac{\beta\gamma(\gamma-1+q) + (\gamma -1)\sigma_1 + q\sigma_2}{(\gamma-1)^2}\Big( 1 + \frac{pq}{(\gamma-1)^2} \Big) \notag \\
\vdots \notag \\
= {} & r_{0}^j a_{0} - \frac{\beta\gamma(\gamma-1+q) + (\gamma -1)\sigma_1 + q\sigma_2}{(\gamma-1)^2} \sum_{i=0}^{j-1} r_{0}^{i},
\end{align}
where $j = 1,2,3,\ldots,$ and $$ r_{0} =  \frac{pq}{(\gamma-1)^2}. $$
We claim that this iteration process must stop after a finite number of steps. To see why, consider the two cases: when $pq \in (0,(\gamma-1)^2]$ and when $pq > (\gamma-1)^2$. \medskip

\noindent{\bf Case 1: } Suppose $pq \in (0,(\gamma-1)^2]$. If $pq = (\gamma-1)^2$, then \eqref{sequence} implies that 
\begin{equation*}
a_{j} = a_{0} - j\frac{\beta\gamma(\gamma-1+q)+(\gamma-1)\sigma_1 + q\sigma_2}{(\gamma-1)^2}.
\end{equation*}
Therefore, $a_j,\,b_j \rightarrow -\infty$ as $j \rightarrow \infty$. If $pq \in (0,(\gamma-1)^2)$, then \eqref{sequence} implies that
\begin{equation*}
a_j = r_{0}^{j}a_0 - \frac{\beta\gamma(\gamma-1+q) + (\gamma-1)\sigma_1 + q\sigma_2}{(\gamma-1)^2}\frac{1-r_{0}^j}{1-r_{0}}.
\end{equation*}
Sending $j \rightarrow \infty$ yields 
$$a_j \rightarrow q_0 < 0 \,\text{ and thus }\, b_j \rightarrow (pq_0 - \sigma_2 - \beta\gamma)/(\gamma-1) < 0.$$ 
In either case, we can find a suitably large $j_0$ such that $a_{j_0},b_{j_0} < 0$ and this implies $u(x),v(x) = \infty$, which is impossible. \medskip

\noindent{\bf Case 2(a): } Let $pq > (\gamma-1)^2$ and $\max \Big\lbrace q_0, p_0 \Big\rbrace > \frac{n-\beta\gamma}{\gamma - 1}$. 

Hereafter, we denote
\begin{equation*}
M = \max \Big\lbrace q_0, p_0 \Big\rbrace.
\end{equation*} 
Let us first assume $M = q_0$. By virtue of \eqref{sequence}, we can find a large $j_0$ such that 
\begin{align*}
a_{j_0} = {} & r_{0}^{j_0}a_0 - \frac{\beta\gamma(\gamma-1+q)+(\gamma-1)\sigma_1 + q\sigma_2}{(\gamma-1)^2}\frac{r_{0}^{j_0}-1}{r_{0}-1} 
= r_{0}^{j_0}(a_0 - M) + M \\
= {} & r_{0}^{j_0}\Big(\frac{n-\beta\gamma}{\gamma-1} - q_0 \Big) + q_0 < 0.
\end{align*}
Thus, $u(x) = \infty$ and we have a contradiction. Likewise, if $pq>(\gamma-1)^2$ but instead $M = p_0$, then we can also apply the previous iteration argument to deduce a contradiction. \medskip

\noindent{\bf Case 2(b): } Let $pq> (\gamma-1)^2$ and $M = \frac{n-\beta\gamma}{\gamma-1}$ where $\gamma \in (1,2]$. Without loss of generality, we assume $M = p_0$.
By H\"{o}lder's inequality,
\begin{align*}
\int_{0}^{R}\int_{B_{t}(x)} |y|^{\sigma_1}v^{q}(y) {} & \,dy \,dt \\
\leq {} & CR^{n-\beta\gamma + 1}\Big( \int_{0}^{R} \Big(\frac{\int_{B_{t}(x)}|y|^{\sigma_1}v^{q}(y)\,dy}{t^{n-\beta\gamma}}\Big)^{\frac{1}{\gamma-1}}\,\frac{dt}{t}\Big)^{\gamma-1}.
\end{align*}
Thus, for $x \in B_{R/4}(0)$,
\begin{align*}
u(x) \geq {} & C\int_{0}^{R} \Big(\frac{\int_{B_{t}(x)}|y|^{\sigma_1}v^{q}(y)\,dy}{t^{n-\beta\gamma}}\Big)^{\frac{1}{\gamma-1}} \,\frac{dt}{t} 
\\ \geq {} & CR^{-\frac{n-\beta\gamma + 1}{\gamma - 1}}\Big( \int_{0}^{R}\int_{B_{t}(x)} |y|^{\sigma_1}v^{q}(y)\,dy \,dt \Big)^{\frac{1}{\gamma-1}} 
\\ \geq {} & CR^{-\frac{n-\beta\gamma + 1}{\gamma - 1}}\Big( \int_{|x| + R/4}^{R}\int_{B_{t}(x)} |y|^{\sigma_1}v^{q}(y)\,dy \,dt \Big)^{\frac{1}{\gamma-1}} \\
\geq {} & CR^{-\frac{n-\beta\gamma}{\gamma - 1}}\Big( \int_{B_{R/4}(0)} |y|^{\sigma_1}v^{q}(y) \,dy\Big)^{\frac{1}{\gamma-1}}.
\end{align*}
As a result, we obtain
\begin{equation}\label{up}
u^{p}(x) \geq CR^{-\frac{n-\beta\gamma}{\gamma - 1}p}\Big( \int_{B_{R/4}(0)} |y|^{\sigma_1}v^{q}(y) \,dy\Big)^{\frac{p}{\gamma-1}}.
\end{equation}
Similarly, we can show that
\begin{equation}\label{vp}
v^{q}(x) \geq CR^{-\frac{n-\beta\gamma }{\gamma - 1}q}\Big( \int_{B_{R/4}(0)} |y|^{\sigma_2}u^{p}(y) \,dy\Big)^{\frac{q}{\gamma-1}}.
\end{equation}
If we multiply \eqref{up} by $|x|^{\sigma_2}$, integrate over $B_{R/4}(0) \backslash B_{\epsilon}(0)$ for suitably small $\epsilon > 0$, apply \eqref{vp} then send $\epsilon \rightarrow 0$, we get
\begin{align*}
\int_{B_{R/4}(0)} {} & |x|^{\sigma_2}u^{p}(x) \,dx \\
\geq {} & \frac{C}{R^{\frac{n-\beta\gamma}{\gamma-1}p -\sigma_2 - n + \frac{n-\beta\gamma}{(\gamma-1)^2 }pq - \frac{\sigma_1 p}{\gamma-1} - \frac{np}{\gamma-1}}} \Big(\int_{B_{R/4}(0)} |x|^{\sigma_2}u^{p}(x) \,dx\Big)^{\frac{pq}{(\gamma-1)^2}} \\
\geq {} & C\Big(\int_{B_{R/4}(0)} |x|^{\sigma_2}u^{p}(x) \,dx\Big)^{\frac{pq}{(\gamma-1)^2}},
\end{align*}
where the above positive constant $C$ is independent of $R$ since
\begin{align*}
\frac{n-\beta\gamma}{\gamma-1}p -  {} & \sigma_2 - n +  \frac{n-\beta\gamma}{(\gamma-1)^2 } pq - \frac{\sigma_1 p}{\gamma-1} - \frac{np}{\gamma-1} \\
= {} & \frac{pq - (\gamma-1)^2}{\gamma-1}\Big\lbrace \frac{n-\beta\gamma}{\gamma-1} - p_0 \Big\rbrace = 0.
\end{align*}
Thus, sending $R \rightarrow \infty$ implies that $|x|^{\sigma_2}u^p (x) \in L^{1}(\mathbb{R}^n)$. If we repeat the previous argument but instead we integrate over $B_{R/4}(0)\backslash B_{R/8}(0)$, then
\begin{equation*}
\int_{B_{R/4}(0)\backslash B_{R/8}(0)} |x|^{\sigma_2}u^{p}(x) \,dx \geq C\Big(\int_{B_{R/4}(0)} |x|^{\sigma_2}u^{p}(x) \,dx\Big)^{\frac{pq}{(\gamma-1)^2}}
\end{equation*}
where $C$ is independent of $R$. Hence, sending $R \rightarrow \infty$ yields
\begin{equation*}
\int_{\mathbb{R}^n} |x|^{\sigma_2}u^{p}(x) \,dx = 0.
\end{equation*}
This implies $u\equiv 0$ and we deduce a contradiction. This completes the proof of the theorem.
\end{proof}

\subsection{Existence of solutions}

\begin{proof}[Proof of Theorem \ref{theorem2}]
Indeed, we find double bounded coefficients with the positive radial solution pair 
\begin{equation*}
u(x) = \frac{1}{(1+|x|^2)^{\theta_1}} \,\text{ and }\, v(x) = \frac{1}{(1+|x|^2)^{\theta_2}},
\end{equation*}
where the rates $\theta_1$ and $\theta_2$ are specified shortly below. In fact, for completeness sake, we provide a solution pair with the slow decay rates and another pair with the fast decay rates.

\noindent{(i)} Choose the slow decay rates:
\begin{equation*}
2\theta_1 = q_0 \,\text{ and }\, 2\theta_2 = p_0,
\end{equation*}
so that $\beta\gamma < 2p\theta_1 - \sigma_2 < n$ and $\beta\gamma < 2q\theta_2 - \sigma_1 < n$. For $|x|\leq R$ with a suitable choice for $R>0$, it is obvious that $u(x)$ and $v(x)$, respectively, are proportional to $W_{\beta,\gamma}(|y|^{\sigma_1}v^q)(x)$ and $W_{\beta,\gamma}(|y|^{\sigma_2}u^p)(x)$. Thus, we may restrict ourselves to the case where $|x|$ is suitably large. Consider the splitting
\begin{equation*}
W_{\beta,\gamma}(|y|^{\sigma_1}v^q)(x) = \Big(\int_{0}^{|x|/2} + \int_{|x|/2}^{\infty}\Big) \Big(\frac{\int_{B_{t}(x)} \frac{|y|^{\sigma_1}}{(1+|y|^2)^{q\theta_2}}\, dy}{t^{n-\beta\gamma}}\Big)^{\frac{1}{\gamma-1}} \,\frac{dt}{t} = H_1 + H_2.
\end{equation*}
Notice that for $y \in B_{t}(x)$,
$$|x|/2 \leq |y| \leq 3|x|/2 \,\text{ whenever }\, |x|/2 \geq t \geq 0,$$ 
and since $ \frac{2\theta_2 q - \sigma_1 - \beta\gamma}{2(\gamma-1)}-\theta_1 = 0, $ there holds
\begin{align*}
H_{1} \geq {} & \frac{1}{C}(1+|x|^2)^{-\frac{q\theta_2 -\sigma_1/2}{\gamma-1}}\int_{0}^{|x|/2} |B_{t}(x)|^{\frac{1}{\gamma -1}} t^{-\frac{n-\beta\gamma}{\gamma-1}} \,\frac{dt}{t}  \\
\geq {} & \frac{1}{C}(1+|x|^2)^{-\frac{q\theta_2 -\sigma_1/2}{\gamma-1}}\int_{0}^{|x|/2} t^{\frac{\beta\gamma}{\gamma-1}} \,\frac{dt}{t}  
\geq \frac{1}{C}(1+|x|^2)^{-\frac{2\theta_2 q - \sigma_1 - \beta\gamma}{2(\gamma-1)} } \geq \frac{1}{C}u(x).
\end{align*}
Similarly, there holds
\begin{equation*}
H_1 \leq C(1+|x|^2)^{-\frac{q\theta_2 -\sigma_1/2}{\gamma-1}}\int_{0}^{|x|/2} t^{\frac{\beta\gamma}{\gamma-1}} \,\frac{dt}{t} \leq C(1+|x|^2)^{-\frac{2\theta_2 q - \sigma_1 - \beta\gamma}{2(\gamma-1)}} \leq Cu(x).
\end{equation*}
Hence, we have that
\begin{equation}\label{H1}
C^{-1}H_{1}  \leq u(x) \leq CH_{1}
\end{equation}
for some positive constant $C$. On the other hand, there holds
\begin{align*}
H_{2} = {} & \int_{|x|/2}^{\infty} \Big(\frac{\int_{B_{t}(x)} \frac{|y|^{\sigma_1}}{(1+|y|^2)^{q\theta_2}}\, dy}{t^{n-\beta\gamma}}\Big)^{\frac{1}{\gamma-1}}\,\frac{dt}{t} \\
\leq {} & C\int_{|x|/2}^{\infty} \Big(\frac{\int_{B_{t + |x|}(0)} |y|^{\sigma_1 - 2q\theta_2}\, dy}{t^{n-\beta\gamma}}\Big)^{\frac{1}{\gamma-1}}\,\frac{dt}{t} \\
\leq {} & C\int_{|x|/2}^{\infty} \Big( \frac{\int_{0}^{t + |x|} r^{n + \sigma_1 - 2\theta_{2}q}\, \frac{dr}{r} }{t^{n-\beta\gamma}} \Big)^{\frac{1}{\gamma-1}}\,\frac{dt}{t} \\
\leq {} & C\int_{|x|/2}^{\infty} t^{\frac{\sigma_1 + \beta\gamma - 2q\theta_2}{\gamma - 1} }\,\frac{dt}{t} 
\leq C(1+|x|^2)^{-\frac{2\theta_2 - \sigma_1 - \beta\gamma}{2(\gamma - 1)} } \leq Cu(x).
\end{align*}
If  $t \geq |x|/2$, we have that $|x|/2 \leq |y| \leq 3|x|/2$ for $y \in B_{|x|/2}(x) \subset B_{t}(x)$ and thus
\begin{align*}
H_{2} \geq {} & \int_{|x|/2}^{\infty} \Big(\frac{\int_{B_{|x|/2}(x)} \frac{|y|^{\sigma_1}}{(1+|y|^2)^{q\theta_2}}\, dy}{t^{n-\beta\gamma}}\Big)^{\frac{1}{\gamma-1}}\,\frac{dt}{t} \\
\geq {} & \frac{1}{C}(1+|x|^2)^{-\frac{2\theta_2 q - \sigma_1 - n}{2(\gamma - 1)} }\int_{|x|/2}^{\infty} t^{-\frac{n-\beta\gamma}{\gamma-1}}\,\frac{dt}{t} \geq \frac{1}{C}u(x).
\end{align*}
Hence, $C^{-1}H_{2} \leq u(x) \leq CH_{2}$ for some positive constant $C$, and by combining this with \eqref{H1}, we obtain
\begin{equation*}
u(x) = c_{1}(x)W_{\beta,\gamma}(|y|^{\sigma_1} v^q)(x)
\end{equation*}
for some double bounded function $c_{1}(x)$. Likewise, similar calculations on the second integral equation will lead to 
\begin{equation*}
v(x) = c_{2}(x)W_{\beta,\gamma}(|y|^{\sigma_{2}} u^p)(x)
\end{equation*}
for some double bounded function $c_{2}(x)$. \medskip

\noindent{(ii)} Choose the fast decay rates: if the stronger condition $p>\frac{(n+\sigma_2)(\gamma-1)}{n-\beta\gamma}$ and $q>\frac{(n+\sigma_1)(\gamma-1)}{n-\beta\gamma}$ hold, then we can take the fast rate $2\theta_1 = 2\theta_2 = \frac{n-\beta\gamma}{\gamma-1}$ in which $2\theta_1 p - \sigma_2 > n$ and $2\theta_2 q - \sigma_1 > n$.
Then 
\begin{align*}
W_{\beta,\gamma}(|y|^{\sigma_1} v^q)(x) = \frac{c_{1}(x)}{(1+|x|^2)^{\frac{n-\beta\gamma}{2(\gamma-1)}}} = c_{1}(x)u(x), \\
W_{\beta,\gamma}(|y|^{\sigma_2} u^p)(x) = \frac{c_{2}(x)}{(1+|x|^2)^{\frac{n-\beta\gamma}{2(\gamma-1)}}} = c_{2}(x)v(x),
\end{align*}
for some double bounded functions $c_{1}(x)$ and $c_{2}(x)$. This follows from direct calculations as before and so we omit the details. This completes the proof of the theorem.
\end{proof}

\section{Slow decay of positive solutions}\label{section3}
In this section, $(u,v)$ is always taken to be a bounded positive solution of system \eqref{Wolff}, and the goal is to prove Theorem \ref{slow theorem}.

\begin{proposition}\label{slow prop 1}
Let $\theta_1 < q_0$ and $\theta_2 < p_0$. Then there does not exist any positive constant $c$ for which
$$\text{ either }\, u(x) \geq c(1+|x|)^{-\theta_1} \,\text{ or }\, v(x) \geq c(1+|x|)^{-\theta_2} \text{ for a.e. } x \in \mathbb{R}^n.$$
\end{proposition}

\begin{proof}
The proof incorporates the iteration scheme found in the proof of Theorem \ref{Liouville}. On the contrary, assume there exists a positive constant $c$ such that 
$$u(x) \geq c(1+|x|)^{-\theta_1}.$$
Indeed, for large $|x|$ and with $\Omega_{x,t} \doteq B_{t-|x|}(0)\backslash B_{\frac{t-|x|}{2}}(0)$, there holds
\begin{equation*}
v(x) \geq c\int_{2|x|}^{\infty} \Big( \int_{\Omega_{x,t}} \frac{ |y|^{\sigma_2}(1+|y|)^{-\theta_1 p} \,dy}{t^{n-\beta\gamma}} \Big)^{\frac{1}{\gamma-1}} \,\frac{dt}{t} \geq c(1+|x|)^{-a_1},
\end{equation*}
where $b_0 = \theta_1$ and $a_1 = \frac{b_0 p - \beta\gamma - \sigma_2}{\gamma-1}$. Inserting this estimate into the first integral equation yields
\begin{equation*}
u(x) \geq c\int_{2|x|}^{\infty} \Big( \int_{\Omega_{x,t}} \frac{ |y|^{\sigma_1}(1+|y|)^{-a_1 q} \,dy}{t^{n-\beta\gamma}} \Big)^{\frac{1}{\gamma-1}} \,\frac{dt}{t} \geq c(1+|x|)^{-b_1},
\end{equation*}
where $b_1 = \frac{a_1 q - \beta\gamma - \sigma_1}{\gamma-1}$. As before, provided that the rates remain positive, we can repeat this procedure inductively to arrive at
$$ v(x) \geq  c(1+|x|)^{-a_j} \,\text{ and }\, u(x) \geq c(1+|x|)^{-b_j},$$
where
$$ a_{j+1} = \frac{pb_j - \beta\gamma - \sigma_2}{\gamma-1} \,\text{ and }\, b_j = \frac{qa_j - \beta\gamma - \sigma_1}{\gamma-1} \,\text{ for }\, j=1,2,3,\ldots.$$
If we set $r_0 = \frac{pq}{(\gamma-1)^2}$ and $\eta_0 = \beta\gamma(\gamma-1+q) + (\gamma-1)\sigma_1 + \sigma_2 q$, we calculate that
\begin{align*}
b_j = {} & \frac{q a_j - \beta\gamma -\sigma_1}{\gamma-1} = \frac{1}{\gamma-1}\Big\lbrace q\frac{pb_{j-1} - \beta\gamma - \sigma_2}{\gamma-1} - \beta\gamma - \sigma_1 \Big\rbrace \\
= {} & \frac{pqb_{j-1} - \eta_0}{(\gamma-1)^2} 
= \frac{1}{(\gamma-1)^2}\Big\lbrace pq\frac{qa_{j-1} - \beta\gamma-\sigma_1}{\gamma-1} - \eta_0 \Big\rbrace \\
= {} & r_{0}^2 b_{j-2} - \frac{\eta_0}{(\gamma-1)^2} ( 1 + r_0 ) = \ldots = r_{0}^j b_{0} - \frac{\eta_0}{(\gamma-1)^2} \sum_{i=0}^{j-1} r_{0}^{i} \\
= {} & r_{0}^j b_{0} - \eta_0 \frac{ r_{0}^j - 1}{pq - (\gamma-1)^2}
= r_{0}^j (\theta_1 - q_0) + q_0.
\end{align*}
Since $\theta_1 < q_0$, this implies that we can choose a suitably large $j_0$ such that $b_{j_0} < 0$. Therefore,
$$ v(x) \geq c\int_{2|x|}^{\infty} t^{\frac{\beta\gamma + \sigma_2 - p b_{j_{0}} }{\gamma-1}} \,\frac{dt}{t} = \infty, $$
but this is impossible. Likewise, if there exists a positive constant $c$ such that $v(x) \geq c(1+|x|)^{-\theta_2}$, then we can apply the same iteration scheme to deduce a contradiction. This completes the proof of the proposition.
\end{proof}

\begin{proposition}\label{slow prop 2}
Let $\theta_3 > q_0$ and $\theta_4 > p_0$ and $(u,v)$ is not integrable. Then there does not exist any positive constant $C$ for which 
$$\text{ either }\, u(x) \leq C(1+|x|)^{-\theta_3} \,\text{ or }\, v(x) \leq C(1+|x|)^{-\theta_4} \text{ for a.e. } x \in \mathbb{R}^n.$$
\end{proposition}

\begin{proof}
Assume there exists a $C>0$ such that $u(x) \leq C(1+|x|)^{-\theta_3}$. Then $n-(n/q_0) \theta_3 < 0$ and for a suitable choice of $R> 0$, we obtain
\begin{align*}
\int_{\mathbb{R}^n} u(x)^{r_0} \,dx = {} & \int_{\mathbb{R}^n} u(x)^{\frac{n}{q_0} } \,dx = \int_{B_{R}(0)} u(x)^{\frac{n}{q_0} } \,dx + \int_{\mathbb{R}^n \backslash B_{R}(0)} u(x)^{\frac{n}{q_0}} \,dx \\
\leq {} & C_{1} + C_{2}\int_{R}^{\infty} t^{n-\frac{n}{q_0} \theta_3} \,\frac{dt}{t} < \infty.
\end{align*}
Similarly, if there exists a $C>0$ such that $v(x) \leq C(1+|x|)^{-\theta_4}$, then $n-(n/p_0) \theta_4 < 0$ and we can show $v \in L^{s_0}(\mathbb{R}^n)$. In any case, we arrive at a contradiction with $(u,v)$ being not integrable.
\end{proof}

\begin{proposition}\label{slow prop 3}
If $(u,v)$ is not integrable but is a decaying solution, then $(u,v)$ necessarily decays with the slow rates as $|x|\rightarrow \infty$.
\end{proposition}

\begin{proof}
This follows immediately from Propositions \ref{slow prop 1} and \ref{slow prop 2}.
\end{proof}

\begin{proof}[Proof of Theorem \ref{slow theorem}]
The theorem follows from Propositions \ref{slow prop 1}--\ref{slow prop 3} .
\end{proof}

\section{Fast decay of positive solutions}\label{section4}
This section contains the proof of Theorem \ref{improved fast theorem} and throughout the section, we assume the same conditions as those stated in the theorem. Moreover, $(u,v)$ is assumed to be an optimal integrable solution of \eqref{Wolff} unless stated otherwise.

\begin{proposition}\label{fast lower bound}
For suitably large $|x|$, there exists a constant $c>0$ such that
$$ u(x),v(x) \geq c|x|^{-\frac{n-\beta\gamma}{\gamma - 1}}.$$
\end{proposition}
\begin{proof}
For suitably large $|x|$, we have that
\begin{equation*}
u(x) \geq c \int_{1+|x|}^{\infty} \Big( \frac{\int_{B_{1}(0)} |y|^{\sigma_{1}} v^{q}(y) \,dy}{t^{n-\beta\gamma}} \Big)^{\frac{1}{\gamma - 1}} \frac{dt}{t} \geq c \int_{1+|x|}^{\infty} t^{\frac{\beta\gamma - n}{\gamma - 1} } \frac{dt}{t} \geq c|x|^{-\frac{n-\beta\gamma}{\gamma - 1}}.
\end{equation*}
The corresponding estimate for $v(x)$ follows similarly.
\end{proof}

\begin{proposition}\label{fast prop 1}
There holds $ u(x) \simeq |x|^{-\frac{n-\beta\gamma}{\gamma - 1} }.$
\end{proposition}

\begin{proof}
By the hypotheses, there is some rate $\theta_1$ for which  $u(x) \simeq |x|^{-\theta_1}$. First, we claim that $\theta_1 \geq \frac{n-\beta\gamma}{\gamma - 1}$; otherwise, we can find $\epsilon > 0$ so that $\theta_1 \leq \frac{n-\beta\gamma}{\gamma - 1 + \epsilon}$. If we set $r = \frac{n(\gamma - 1 + \epsilon)}{n - \beta\gamma}$, then for a sufficiently large $R> 0$,
\begin{equation*}
\|u\|_{L^{r}(\mathbb{R}^n)}^r \geq c\int_{\mathbb{R}^n \backslash B_{R}(0)} |x|^{-r \theta_1} \,dx \geq c\int_{R}^{\infty} t^{n - \theta_{1}\frac{n(\gamma - 1 + \epsilon)}{n-\beta\gamma}} \frac{dt}{t} = \infty.
\end{equation*}
This is impossible since \eqref{optimal integrability} ensures that $u \in L^{r}(\mathbb{R}^n)$, and this proves the claim. Then Proposition \ref{fast lower bound} further implies that $$\theta_1 = \frac{n-\beta\gamma}{\gamma - 1}.$$
\end{proof}

\begin{proposition}\label{fast prop 2}
If $p(\frac{n-\beta\gamma}{\gamma - 1}) - \sigma_2 > n$, then $ v(x) \simeq |x|^{-\frac{n-\beta\gamma}{\gamma - 1} }.$
\end{proposition}

\begin{proof}
As before, there is some positive rate $\theta_2$ such that $v(x) \simeq |x|^{-\theta_2}$. By the integrability of $v$, we claim that
$$\theta_2 \geq min\Big\lbrace \frac{n-\beta\gamma}{\gamma - 1},\, \frac{p(\frac{n-\beta\gamma}{\gamma - 1}) - (\beta\gamma + \sigma_2)}{\gamma - 1} \Big\rbrace.$$
However, notice that 
$$min\Big\lbrace \frac{n-\beta\gamma}{\gamma - 1},\, \frac{p(\frac{n-\beta\gamma}{\gamma - 1}) - (\beta\gamma + \sigma_2)}{\gamma - 1} \Big\rbrace =  \frac{n-\beta\gamma}{\gamma - 1},$$
since $p(\frac{n-\beta\gamma}{\gamma - 1}) - \sigma_2 > n$. Assume the contrary. Then we can find $\epsilon > 0$ such that 
$$\theta_2 \leq \frac{n-\beta\gamma}{\gamma  - 1 + \epsilon} \,\text{ and set }\, s = \frac{n(\gamma  - 1 + \epsilon)}{n-\beta\gamma}.$$
We can choose $R>0$ suitably large so that
\begin{equation*}
\|v\|_{L^{s}(\mathbb{R}^n)}^s \geq c\int_{\mathbb{R}^n \backslash B_{R}(0)} |x|^{-s \theta_2} \,dx \geq c\int_{R}^{\infty} t^{n - \theta_{2}\frac{n(\gamma - 1 + \epsilon)}{n-\beta\gamma}} \frac{dt}{t} = \infty,
\end{equation*}
but this is impossible since \eqref{optimal integrability} ensures that $v \in L^{s}(\mathbb{R}^n)$. This proves the claim and by combining this with Proposition \ref{fast lower bound}, we obtain that 
$$\theta_2 = \frac{n-\beta\gamma}{\gamma - 1}. $$
\end{proof}

\begin{proposition}\label{fast prop 3}
If $p(\frac{n-\beta\gamma}{\gamma - 1}) - \sigma_2 = n$, then $ v(x) \simeq |x|^{-\frac{n-\beta\gamma}{\gamma - 1} }(\ln |x|)^{\frac{1}{\gamma - 1}}.$
\end{proposition}

\begin{proof}
We shall make use of the following identity which follows from elementary arguments from calculus (see \cite{Lei2011}). For $\lambda > 0$,
\begin{equation}\label{lambda limit}
\lim_{|x|\rightarrow \infty} \frac{|x|^{\frac{n-\beta\gamma}{\gamma - 1} } }{ (\ln \lambda |x|)^{\frac{1}{\gamma - 1}} } \int_{\lambda |x|}^{\infty} \Big( \frac{\ln t}{t^{n-\beta\gamma}} \Big)^{\frac{1}{\gamma - 1}} \frac{dt}{t} = \frac{\gamma - 1}{n-\beta\gamma}\lambda^{-\frac{n-\beta\gamma}{\gamma - 1}}.
\end{equation}
For $\lambda \in (1/2, 1)$, we write
$$ v(x) \leq C\Big( \int_{0}^{\lambda |x|} + \int_{\lambda |x|}^{\infty} \Big) \Big( \frac{\int_{B_{t}(x)} |y|^{\sigma_2}u^{p}(y) \,dy}{t^{n-\beta\gamma}} \Big)^{\frac{1}{\gamma-1}} \,\frac{dt}{t} = C(I_1 + I_2). $$
For large $|x|$, since $u(x) \simeq |x|^{-\frac{n-\beta\gamma}{\gamma - 1}}$ and  
$$ p\frac{n-\beta\gamma}{(\gamma - 1)^2} - \frac{\sigma_2 + \beta\gamma}{\gamma - 1} = \frac{n-\beta\gamma}{\gamma - 1},$$
we have that
$$ I_1 \leq C |x|^{-p\frac{n-\beta\gamma}{(\gamma - 1)^2} + \frac{\sigma_2}{\gamma - 1}} \int_{0}^{\lambda |x|} t^{\frac{\beta\gamma}{\gamma - 1} } \frac{dt}{t} \leq C|x|^{-p\frac{n-\beta\gamma}{(\gamma - 1)^2} + \frac{\sigma_2 + \beta\gamma}{\gamma - 1}} \leq C|x|^{-\frac{n-\beta\gamma}{\gamma - 1}}. $$
Thus,
\begin{equation}\label{I1}
\lim_{|x|\rightarrow \infty} |x|^{\frac{n-\beta\gamma}{\gamma - 1}}(\ln |x|)^{-\frac{1}{\gamma - 1}}I_1 = 0.
\end{equation}
Likewise, choose a sufficiently large $R>0$. Then for large $|x|$ we have that
\begin{align*}
I_2 \leq {} & C\int_{\lambda |x|}^{\infty} \Big( \frac{\int_{B_{t}(x)} |y|^{\sigma_2}u^{p}(y) \,dy}{t^{n-\beta\gamma}} \Big)^{\frac{1}{\gamma-1}} \,\frac{dt}{t} \\
\leq {} & C\int_{\lambda |x|}^{\infty} \Big( \frac{\int_{B_{R}(0) } |y|^{\sigma_2}u^{p}(y) \,dy + \int_{B_{t+|x|}(0) \backslash B_{R}(0) } |y|^{\sigma_2}u^{p}(y) \,dy}{t^{n-\beta\gamma}} \Big)^{\frac{1}{\gamma-1}} \,\frac{dt}{t} \\
\leq {} & C\int_{\lambda |x|}^{\infty} \Big( \frac{C_1 + C_2 \int_{1}^{t + |x|} r^{n- \sigma_2 -p(\frac{n-\beta\gamma}{\gamma - 1})} \frac{dr}{r} }{t^{n-\beta\gamma}} \Big)^{\frac{1}{\gamma-1}} \,\frac{dt}{t} \\
\leq {} & C\int_{\lambda |x|}^{\infty} \Big( \frac{C_1 + C_2 \int_{1}^{t + |x|} r^{-1} \,dr }{t^{n-\beta\gamma}} \Big)^{\frac{1}{\gamma-1}} \,\frac{dt}{t} 
\leq C\int_{\lambda |x|}^{\infty} \Big( \frac{\ln t}{t^{n-\beta\gamma}} \Big)^{\frac{1}{\gamma-1}} \,\frac{dt}{t}.
\end{align*}
Combining this estimate with identity \eqref{lambda limit} and sending $\lambda \rightarrow 1$ yields
\begin{equation}\label{I2}
\lim_{|x| \rightarrow \infty} \frac{|x|^{\frac{n-\beta\gamma}{\gamma - 1}} }{(\ln |x|)^{\frac{1}{\gamma - 1} } } I_2 \leq C.
\end{equation}
Hence, \eqref{I1} and \eqref{I2} imply
\begin{equation}\label{I3}
\lim_{|x| \rightarrow \infty} \frac{|x|^{\frac{n-\beta\gamma}{\gamma - 1}} }{(\ln |x|)^{\frac{1}{\gamma - 1} } } v(x) \leq C.
\end{equation}
Notice that for any $\lambda > 1$, $B_{t - |x|}(0) \subset B_{t}(x)$ if $t > \lambda|x|$. For a proper choice of $R>0$, Proposition \ref{fast lower bound} implies that
\begin{align*}
v(x) \geq {} & c\int_{\lambda |x|}^{\infty} \Big( \frac{\int_{B_{t-|x|}(0)\backslash B_{R}(0) } |y|^{\sigma_2}u^{p}(y) \,dy}{t^{n-\beta\gamma}} \Big)^{\frac{1}{\gamma-1}} \,\frac{dt}{t} \\
\geq {} & c\int_{\lambda |x|}^{\infty} \Big( \frac{\int_{R}^{t - |x|} r^{n+\sigma_2 - p(\frac{n-\beta\gamma}{\gamma - 1}) } \frac{dr}{r}}{t^{n-\beta\gamma}} \Big)^{\frac{1}{\gamma-1}} \,\frac{dt}{t} \\
\geq {} & c\int_{\lambda |x|}^{\infty} \Big( \frac{\ln t}{t^{n-\beta\gamma}} \Big)^{\frac{1}{\gamma-1}} \,\frac{dt}{t}.
\end{align*}
Thus, applying identity \eqref{lambda limit} to this then sending $\lambda \rightarrow 1$ yields
\begin{equation}\label{I4}
\lim_{|x| \rightarrow \infty} \frac{|x|^{\frac{n-\beta\gamma}{\gamma - 1}} }{(\ln |x|)^{\frac{1}{\gamma - 1} } } v(x) \geq c > 0.
\end{equation}
Hence, \eqref{I3} and \eqref{I4} imply the desired result.
\end{proof}

\begin{proposition}\label{fast prop 4}
If $p(\frac{n-\beta\gamma}{\gamma - 1}) - \sigma_2 < n$, then $ v(x) \simeq |x|^{-\frac{p(\frac{n-\beta\gamma}{\gamma - 1}) - (\beta\gamma + \sigma_2)}{\gamma - 1}}.$
\end{proposition}

\begin{proof}
Fix some $R>0$ and for $|x| > 2R$, write
$$ v(x) \leq C\Big( \int_{\Omega_1} + \int_{\Omega_{1}^{c}} \Big) \Big( \frac{\int_{B_{t}(x)} |y|^{\sigma_2}u^{p}(y) \,dy}{t^{n-\beta\gamma}} \Big)^{\frac{1}{\gamma-1}} \,\frac{dt}{t} = C(J_1 + J_2),$$
where $\Omega_{1} = [|x| - R, |x| + R]$. Then we claim that
\begin{enumerate}[(a)]
\item $\lim_{|x| \rightarrow \infty} |x|^{\frac{p(\frac{n-\beta\gamma}{\gamma - 1}) - (\beta\gamma + \sigma_2)}{\gamma - 1}}J_1 = 0$,

\item $\lim_{|x| \rightarrow \infty} |x|^{\frac{p(\frac{n-\beta\gamma}{\gamma - 1}) - (\beta\gamma + \sigma_2)}{\gamma - 1}}J_2 = C$,
\end{enumerate}
and the desired result follows once we prove this.

(a) Indeed, $B_{t}(x) \subset B_{2t + R}(0)$ if $t \in \Omega_1$ and so Proposition \ref{fast prop 1} yields
$$J_1 \leq C \int_{\Omega_1} \Big( \frac{\int_{0}^{2t + R} r^{n + \sigma_2 - p(\frac{n-\beta\gamma}{\gamma - 1}) } \frac{dr}{r} }{t^{n-\beta\gamma}} \Big)^{\frac{1}{\gamma-1}} \,\frac{dt}{t} \leq C|x|^{-\frac{p(\frac{n-\beta\gamma}{\gamma - 1}) - (\beta\gamma + \sigma_2)}{\gamma - 1} - 1}$$
and part (a) follows accordingly.

(b) To show this part, we first verify that
\begin{equation}\label{J34}
 \int_{0}^{\infty} \Big( \frac{\int_{B_{t}(e)} |y|^{\sigma_2 - p(\frac{n-\beta\gamma}{\gamma - 1}) } \,dy }{t^{n-\beta\gamma}} \Big)^{\frac{1}{\gamma-1}} \,\frac{dt}{t} < \infty,
\end{equation}
where $e$ is any unit vector. To show this, let $c \in (0,1)$ and consider the splitting 
$$  J_3 + J_4 = \Big( \int_{0}^{c} + \int_{c}^{\infty} \Big) \Big( \frac{\int_{B_{t}(e)} |y|^{\sigma_2 - p(\frac{n-\beta\gamma}{\gamma - 1}) } \,dy }{t^{n-\beta\gamma}} \Big)^{\frac{1}{\gamma-1}} \,\frac{dt}{t}.$$
Indeed,
\begin{align*}
J_{3} \leq {} & \int_{0}^{c} \Big( \frac{\int_{B_{t}(e)} |y|^{\sigma_2 - p(\frac{n-\beta\gamma}{\gamma - 1}) } \,dy }{t^{n-\beta\gamma}} \Big)^{\frac{1}{\gamma-1}} \,\frac{dt}{t} 
\leq C \int_{0}^{c} \Big( \frac{ |B_{t}(e)|}{t^{n-\beta\gamma}} \Big)^{\frac{1}{\gamma-1}} \,\frac{dt}{t} \\
\leq {} & \int_{0}^{c} t^{\frac{\beta\gamma}{\gamma - 1}} \frac{dt}{t} < \infty,
\end{align*}
since $y \in B_{t}(e)$ ensures that $1 - c < |y| < 1+c$. We can find a suitably large $R> 0$ so that
\begin{align*}
J_{4} \leq {} & C\int_{c}^{\infty} \Big( \frac{\int_{B_{Rt}(0)} |y|^{\sigma_2 - p(\frac{n-\beta\gamma}{\gamma - 1}) } \,dy }{t^{n-\beta\gamma}} \Big)^{\frac{1}{\gamma-1}} \,\frac{dt}{t} \\
\leq {} & C\int_{c}^{\infty} \Big( \frac{\int_{0}^{Rt} r^{n + \sigma_2 - p(\frac{n-\beta\gamma}{\gamma - 1}) } \frac{dr}{r} }{t^{n-\beta\gamma}} \Big)^{\frac{1}{\gamma-1}} \,\frac{dt}{t} \\
\leq {} & C\int_{c}^{\infty} t^{-\frac{p(\frac{n-\beta\gamma}{\gamma - 1}) - (\beta\gamma + \sigma_2)}{\gamma - 1}} \,\frac{dt}{t} < \infty.
\end{align*}
This completes the proof of the claim. Now we turn our attention to estimating the term $J_2$. Indeed, by setting $$\Omega_2 = [1 - R/|x|, 1 + R/|x|],$$ using the change of variables $$z = \frac{y}{|x|},\, s = \frac{t}{|x|},$$ and applying \eqref{J34}, we get that
\begin{align*}
J_2 \leq {} & C\int_{\Omega_{1}^{c}} \Big( \frac{\int_{B_{t}(x)} |y|^{-p(\frac{n-\beta\gamma}{\gamma-1}) + \sigma_2} \,dy}{t^{n-\beta\gamma}} \Big)^{\frac{1}{\gamma-1}} \,\frac{dt}{t} \\
\leq {} & C \int_{\Omega_{2}^{c}} \Big( \frac{\int_{B_{s}(x/|x|)} |z|^{-p(\frac{n-\beta\gamma}{\gamma-1})+\sigma_2} \,dz}{s^{n-\beta\gamma}} \Big)^{\frac{1}{\gamma-1}} \,\frac{ds}{s} |x|^{-\frac{p(\frac{n-\beta\gamma}{\gamma-1}) - (\beta\gamma + \sigma_2)}{\gamma-1}} \\
\leq {} & C|x|^{-\frac{p(\frac{n-\beta\gamma}{\gamma-1}) - (\beta\gamma + \sigma_2)}{\gamma-1}}.
\end{align*}
Likewise, using Proposition \ref{fast lower bound} combined with similar arguments as above, we can also show that
$$J_{2} \geq  c|x|^{-\frac{p(\frac{n-\beta\gamma}{\gamma-1}) - (\beta\gamma + \sigma_2)}{\gamma-1}}.$$
This proves part (b) and thus completes the proof of the proposition.
\end{proof}

\begin{proposition}\label{fast prop 5}
Suppose that $(u,v)$ is a bounded and decaying solution of system \eqref{Wolff}. If $(u,v)$ decays with the fast rates as $|x| \rightarrow \infty$, then it is an optimal integrable solution, i.e., $(u,v) \in L^{r}(\mathbb{R}^n) \times L^{s}(\mathbb{R}^n)$ for all $(r,s)$ satisfying condition \eqref{optimal integrability}. 
\end{proposition}
\begin{proof}
Suppose that $(u,v)$ decays with the fast rates as $|x| \rightarrow \infty$ and $(r,s)$ satisfies \eqref{optimal integrability}. 

\begin{enumerate}[(i)]
\item If $u(x)$ decays with the rate $\frac{n-\beta\gamma}{\gamma - 1}$, then we can find a suitably large $R>0$ so that
\begin{align*}
\int_{\mathbb{R}^n} u(x)^{r} \,dx \leq {} & \int_{B_{R}(0)} u(x)^{r} \,dx  + \int_{\mathbb{R}^n \backslash B_{R}(0)} u(x)^{r} \,dx \\
\leq {} & C_{1} + C_{2}\int_{\mathbb{R}^n \backslash B_{R}(0)} |x|^{-r\frac{n-\beta\gamma}{\gamma - 1}} \,dx
\leq C\int_{R}^{\infty} t^{n - r\frac{n-\beta\gamma}{\gamma - 1}} \, \frac{dt}{t} < \infty.
\end{align*}
Likewise, if $v(x)$ decays with the rate $\frac{n-\beta\gamma}{\gamma - 1}$, we can use similar arguments to show $v \in L^{s}(\mathbb{R}^n)$. 

\item Assume $v(x) \simeq |x|^{-\frac{n-\beta\gamma}{\gamma - 1}} (\ln |x|)^{\frac{1}{\gamma - 1}}$. For small $\epsilon > 0$, we can find $R>0$ such that
$$ (\ln |x|)^{\frac{s}{\gamma - 1}} \leq C|x|^{\epsilon} \,\text{ for }\, |x| > R. $$
Thus
$$ \int_{\mathbb{R}^n} v(x)^{s} \,dx \leq C_{1} + C_{2}\int_{R}^{\infty} t^{n - s\frac{n-\beta\gamma}{\gamma - 1} + \epsilon} \,\frac{dt}{t} < \infty,$$
since $n - s\frac{n-\beta\gamma}{\gamma - 1} + \epsilon < 0$ provided that $\epsilon$ is chosen to be small enough.

\item Assume
$$v(x) \simeq |x|^{-p\frac{n-\beta\gamma}{(\gamma - 1)^2} + \frac{\sigma_2 + \beta\gamma}{\gamma - 1}}.$$
Condition \eqref{optimal integrability} ensures that
$$n - \frac{s}{\gamma - 1}\Big( p(\frac{n-\beta\gamma}{\gamma - 1}) - (\beta\gamma + \sigma_2) \Big) < 0,$$
and thus
$$\int_{\mathbb{R}^n} v(x)^{s} \,dx \leq C_{1} + C_{2} \int_{R}^{\infty} t^{n - \frac{s}{\gamma - 1}( p(\frac{n-\beta\gamma}{\gamma - 1}) - (\beta\gamma + \sigma_2) )} \,\frac{dt}{t} < \infty.$$
\end{enumerate}
In any case, $(u,v) \in L^{r}(\mathbb{R}^n) \times L^{s}(\mathbb{R}^n)$, and this completes the proof.

\end{proof}

\begin{proof}[Proof of Theorem \ref{improved fast theorem}]
This follows directly from Propositions \ref{fast prop 1} to \ref{fast prop 5}.
\end{proof}

\section{Systems of quasilinear differential equations}\label{section5}

\begin{proof}[Proof of Corollary \ref{quasilinear Liouville}]

On the contrary, assume that $(u,v)$ is a positive solution of system \eqref{quasilinear system} satisfying 
$$\inf_{\mathbb{R}^n} u = \inf_{\mathbb{R}^n} v = 0.$$ 
Since the coefficients are double bounded, the global estimate \eqref{KM} ensures there is a positive constant $C$ such that
\begin{align*}
C^{-1}W_{1,\gamma}(|y|^{\sigma_1}v^{q})(x) \leq u(x) \leq CW_{1,\gamma}(|y|^{\sigma_1}v^{q})(x),\\ \medskip
C^{-1}W_{1,\gamma}(|y|^{\sigma_2}u^{p})(x) \leq v(x) \leq CW_{1,\gamma}(|y|^{\sigma_2}u^{p})(x).
\end{align*}
From this, we clearly get two double bounded functions $c_{1}(x)$ and $c_{2}(x)$ such that
\begin{align*}
u(x) = c_{1}(x)W_{1,\gamma}(|y|^{\sigma_1}v^q)(x),\\ \medskip
v(x) = c_{2}(x)W_{1,\gamma}(|y|^{\sigma_2}u^p)(x),
\end{align*}
but this is impossible in view of Theorem \ref{Liouville}.
\end{proof}

\begin{proof}[Proofs of Corollaries \ref{slow quasilinear} and \ref{fast quasilinear improved}] 

Indeed, by virtue of estimate \eqref{KM}, there exist  double bounded coefficients $c_{1}(x)$ and $c_{2}(x)$ such that $(u,v)$ is a positive solution of the integral system
\begin{align*}
u(x) = c_{1}(x)W_{1,\gamma}(|y|^{\sigma_1}v^q)(x),\\ \medskip
v(x) = c_{2}(x)W_{1,\gamma}(|y|^{\sigma_2}u^p)(x).
\end{align*}
Then the results follow directly from Theorem \ref{slow theorem} and Theorem \ref{improved fast theorem}.
\end{proof}

\noindent{\bf Acknowledgment:} The author wishes to thank the anonymous referees of this paper. Their comments and suggestions have certainly improved the presentation and overall quality of this manuscript. Part of this work was completed while the author was visiting the School of Mathematical Sciences at the University of Science and Technology of China in Hefei, PRC and the Department of Applied Mathematics at Northwestern Polytechnical University in Xi'an, PRC. The author wishes to thank both institutions for their support and hospitality.


\begin{thebibliography}{10}

\bibitem{CGS89}
L.~Caffarelli, B.~Gidas, and J.~Spruck.
\newblock Asymptotic symmetry and local behavior of semilinear elliptic equations with critical {S}obolev growth.
\newblock {\em Comm. Pure Appl. Math.}, 42 (3) (1989) 271--297.

\bibitem{CKN84}
L.~Caffarelli, R.~Kohn, and L.~Nirenberg.
\newblock First order interpolation inequalities with weights.
\newblock {\em Compos. Math.}, 53 (3) (1984) 259--275.

\bibitem{Caristi2008}
G.~Caristi, L.~D'Ambrosio, and E.~Mitidieri.
\newblock Representation formulae for solutions to some classes of higher order systems and related {L}iouville theorems.
\newblock {\em Milan J. Math.}, 76 (1) (2008) 27--67.

\bibitem{ChenLu2014}
H.~Chen and Z.~L\"{u}.
\newblock The properties of positive solutions to an integral system involving Wolff potential.
\newblock {\em Discrete Contin. Dyn. Syst.}, 34 (5) (2014) 1879--1904.

\bibitem{ChengLi15}
Z.~Cheng and C.~Li.
\newblock Shooting method with sign-changing nonlinearity.
\newblock {\em Nonlinear Anal.}, 114 (2015) 2--12.

\bibitem{CL11}
W.~Chen and C.~Li.
\newblock Radial symmetry of solutions for some integral systems of {W}olff type.
\newblock {\em Discrete Contin. Dyn. Syst.}, 30 (4) (2011) 1083--1093.

\bibitem{CL13}
W.~Chen and C.~Li.
\newblock Super polyharmonic property of solutions for {PDE} systems and its applications.
\newblock {\em Commun. Pure Appl. Anal.}, 12 (6) (2013) 2497--2514.

\bibitem{CLO05}
W.~Chen, C.~Li, and B.~Ou.
\newblock Classification of solutions for a system of integral equations.
\newblock {\em Comm. Partial Differential Equations}, 30 (2005) 59--65.


\bibitem{CLO06}
W.~Chen, C.~Li, and B.~Ou.
\newblock Classification of solutions for an integral equation.
\newblock {\em Comm. Pure Appl. Math.}, 59 (3) (2006) 330--343.

\bibitem{DAmbrosioMitidieri14}
L.~D'Ambrosio and E.~Mitidieri.
\newblock {H}ardy-{L}ittlewood-{S}obolev systems and related {L}iouville theorems.
\newblock {\em Discrete Contin. Dyn. Syst. Ser. S}, 7 (4) (2014) 653--671.

\bibitem{DdPMW08}
J.~D\'{a}vila, M.~del Pino, M.~Musso and J.~Wei.
\newblock Fast and slow decay solutions for supercritical elliptic problems in exterior domains.
\newblock {\em Calc. Var. Partial Differential Equations}, 32 (4) (2008) 453--480.

\bibitem{DouZhou15}
J.~Dou and H.~Zhou.
\newblock Liouville theorems for fractional H\'{e}non equation and system in $\mathbb{R}^n$.
\newblock {\em Commun. Pure Appl. Anal.}, 14 (5) (2015) 1915--1927.

\bibitem{GS81a}
B.~Gidas and J.~Spruck.
\newblock A priori bounds for positive solutions of nonlinear elliptic equations.
\newblock {\em Comm. Partial Differential Equations.}, 6 (8) (1981) 883--901.

\bibitem{GS81}
B.~Gidas and J.~Spruck.
\newblock Global and local behavior of positive solutions of nonlinear elliptic equations.
\newblock {\em Comm. Pure Appl. Math.}, 34 (4) (1981) 525--598.


\bibitem{Henon73}
M.~H\'{e}non.
\newblock Numerical experiments on the stability of spherical stellar systems.
\newblock {\em Astronom. Astrophys.}, 24 (1973) 229--238.

\bibitem{KM94}
T.~Kilpel\"{a}inen and J.~Mal\'{y}.
\newblock The {W}iener test and potential estimates for quasilinear elliptic equations.
\newblock {\em Acta Math.}, 172 (1) (1994) 137--161.

\bibitem{Lei2011}
Y.~Lei.
\newblock Decay rates of solutions of an integral system of {W}olff type.
\newblock {\em Potential Anal.}, 35 (4) (2011) 387--402.

\bibitem{Lei2013}
Y.~Lei.
\newblock Asymptotic properties of positive solutions of the {H}ardy-{S}obolev type equations.
\newblock {\em J. Differential Equations}, 254 (4) (2013) 1774--1799.

\bibitem{LeiLi2012}
Y.~Lei and C.~Li.
\newblock Integrability and asymptotics of positive solutions of a $\gamma$-{L}aplace system.
\newblock {\em J. Differential Equations}, 252 (3) (2012) 2739--2758.

\bibitem{LL13a}
Y.~Lei and C.~Li.
\newblock Decay properties of the {H}ardy-{L}ittlewood-{S}obolev systems of the {L}ane-{E}mden type (2013).
\newblock {\em http://arXiv:1302.5567}.

\bibitem{LL13}
Y.~Lei and C.~Li.
\newblock Sharp criteria of {L}iouville type for some nonlinear systems.
\newblock {\em Discrete Contin. Dyn. Syst}, 36 (6) (2016) 3277--3315.

\bibitem{LLM12}
Y.~Lei, C.~Li, and C.~Ma.
\newblock Asymptotic radial symmetry and growth estimates of positive solutions to weighted {H}ardy-{L}ittlewood-{S}obolev system of integral equations.
\newblock {\em Calc. Var. Partial Differential Equations}, 45 (1-2) (2012) 43--61.

\bibitem{YLi92}
Y.~Li.
\newblock Asymptotic behavior of positive solutions of equation {$\Delta u + K(x)u^p = 0$} in $\mathbb{R}^n$.
\newblock {\em J. Differential Equations}, 95 (2) (1992) 304--330.

\bibitem{Lieb83}
E.~Lieb.
\newblock Sharp constants in the {Hardy--Littlewood--Sobolev} and related inequalities.
\newblock {\em Ann. of Math.}, 118 (2) (1983) 349--374.

\bibitem{LuZhu11}
G. Lu and J. Zhu.
\newblock Symmetry and regularity of extremals of an integral equation related to the {H}ardy-{S}obolev inequality.
\newblock {\em Calc. Var. Partial Differential Equations}, 42 (3-4) (2011) 563--577.

\bibitem{Mitidieri96}
E.~Mitidieri.
\newblock Nonexistence of positive solutions of semilinear elliptic systems in ${R}^{N}$.
\newblock {\em Differ. Integral Equations}, 9 (3) (1996) 465--480.

\bibitem{MitidieriPohozaev01}
E.~Mitidieri and S.~I. Pokhozhaev.
\newblock A priori estimates and the absence of solutions of nonlinear partial differential equations and inequalities.
\newblock {\em Tr. Mat. Inst. Steklova}, 234 (2001) 1--384.


\bibitem{Phan12}
Q.~H. Phan.
\newblock Liouville-type theorems and bounds of solutions for {H}ardy-{H}\'{e}non systems.
\newblock {\em Adv. Differential Equations}, 17 (7-8) (2012) 605--634.


\bibitem{PhucVerbitsky2008}
N. C.~Phuc and I. E.~Verbitsky.
\newblock Quasilinear and Hessian equations of {L}ane-{E}mden type.
\newblock {\em Ann. of Math.}, 168 (3) (2008) 859--914.

\bibitem{PQS07}
P.~Pol{\'a}{\v{c}}ik, P.~Quittner, and P.~Souplet.
\newblock Singularity and decay estimates in superlinear problems via {L}iouville-type theorems, {I}: {E}lliptic equations and systems.
\newblock {\em Duke Math. J.}, 139 (3) (2007) 555--579.

\bibitem{SZ96}
J.~Serrin and H.~Zou.
\newblock Non-existence of positive solutions of {L}ane-{E}mden systems.
\newblock {\em Differ. Integral Equations}, 9 (4) (1996) 635--653.


\bibitem{Souplet09}
P.~Souplet.
\newblock The proof of the {L}ane--{E}mden conjecture in four space dimensions.
\newblock {\em Adv. Math.}, 221 (5) (2009) 1409--1427.

\bibitem{SunLei2012}
S.~Sun and Y.~Lei.
\newblock Fast decay estimates for integrable solutions of the {L}ane-{E}mden type integral systems involving the {W}olff potentials.
\newblock {\em J. Funct. Anal.}, 263 (12) (2012) 3857--3882.

\bibitem{Villavert:14b}
J.~Villavert.
\newblock Shooting with degree theory: {A}nalysis of some weighted poly-harmonic systems.
\newblock {\em J. Differential Equations}, 257 (4) (2014) 1148--1167.

\bibitem{Villavert:14e}
J.~Villavert.
\newblock A characterization of fast decaying solutions for quasilinear and {W}olff type systems with singular coefficients.
\newblock {\em J. Math. Anal. Appl.}, 424 (2) (2015) 1348--1373.

\bibitem{Villavert:14d}
J.~Villavert.
\newblock Qualitative properties of solutions for an integral system related to the {H}ardy--{S}obolev inequality.
\newblock {\em J. Differential Equations}, 258 (5) (2015) 1685--1714.

\bibitem{Villavert:14c}
J.~Villavert.
\newblock {Sharp existence criteria for positive solutions of {H}ardy-{S}obolev type systems}.
\newblock {\em Commun. Pure Appl. Anal.}, 14 (2) (2015) 493--515.

\bibitem{JYang15}
J.~Yang.
\newblock {Fractional Sobolev--Hardy inequality in $\mathbb{R}^N$}.
\newblock {\em Nonlinear Anal.}, 119 (2015) 179--185.
\end{thebibliography}

\end{document}